\documentclass[11pt,twoside]{amsart}
\usepackage[latin1]{inputenc}
\usepackage{verbatim}
\usepackage{amsmath}
\usepackage{amsfonts}
\usepackage{amssymb}
\usepackage{amsthm}
\usepackage{color,soul}
\usepackage[usenames]{xcolor}
\usepackage{graphicx}
\usepackage[all]{xy}
\usepackage{fullpage}
\usepackage{hyperref}
\usepackage{multicol}
\usepackage{listings}
\usepackage{verbatim}
\usepackage{endnotes}
\usepackage{enumitem}
\usepackage[abs]{overpic}
\usepackage[normalem]{ulem} 
\usepackage{MnSymbol} 
\usepackage{tikz}
\usetikzlibrary{cd}
\usetikzlibrary{backgrounds}
\usetikzlibrary{positioning}
\usetikzlibrary{arrows}

\usetikzlibrary{shapes,plotmarks,matrix}
\usepackage[balance,small,nohug]{diagrams}

\newcommand{\Pic}{\mathrm{Pic}}
\newcommand{\Mov}{\mathrm{Mov}}

\newcommand{\PP}{\mathbb{P}}
\newcommand{\FF}{\mathbb{F}}
\newcommand{\EFF}{\overline{\mathrm{Eff}}}

\theoremstyle{definition}
\newtheorem{defn}{Definition}[section]
\newtheorem*{defin}{Definition}
\newtheorem{ex}[defn]{Example}
\newtheorem{rmk}[defn]{Remark}
\theoremstyle{plain}
\newtheorem{thm}[defn]{Theorem}
\newtheorem{lem}[defn]{Lemma}
\newtheorem{cor}[defn]{Corollary}
\newtheorem{prop}[defn]{Proposition}

\newtheorem{conj}[defn]{Conjecture}

\newtheorem*{thm1}{Theorem A}
\newtheorem*{thm2}{Theorem B}
\newtheorem*{cor3}{Corollary C}

\begin{document}
\title{On the birational geometry of Hilbert schemes of points and Severi divisors}
\author{C\'esar Lozano Huerta}
\author{Tim Ryan}

\address{Universidad Nacional Aut\'onoma de M\'exico\\
Instituto de Matem\'aticas \\
Oaxaca, Mex.}
\email{lozano@im.unam.mx} 

\address{Stony Brook University\\
Department of mathematics \\
NY, USA.}
\email{timothy.ryan@stonybrook.edu} 



\begin{abstract}
We study the birational geometry of Hilbert schemes of points on non-minimal surfaces. In particular,  we study the weak Lefschetz Principle in the context of birational geometry. We focus on the interaction of the stable base locus decomposition (SBLD) of the cones of effective divisors of $X^{[n]}$ and $Y^{[n]}$, when there is a birational morphism $f:X\rightarrow Y$ between surfaces. In this setting, $N^1(Y^{[n]})$ embeds in $N^1(X^{[n]})$, and we ask if the restriction of the stable base locus decomposition of $N^1(X^{[n]})$ yields the respective decomposition in $N^1(Y^{[n]})$ $i.e.$, if the weak Lefschetz Principle holds. Even though the stable base loci in $N^1(X^{[n]})$ fails to provide information about how the two decompositions interact, we show that the restriction of the augmented stable base loci of $X^{[n]}$ to $Y^{[n]}$ is equal to the stable base locus decomposition of $Y^{[n]}$. We also exhibit effective divisors induced by Severi varieties. We compute the classes of such divisors and observe that in the case that $X$ is the projective plane,  these divisors yield walls of the SBLD for some cases.
\end{abstract}

\maketitle

\section{Introduction} 

\noindent
\noindent
The study of how points on surfaces move and form families has driven research in algebraic geometry for a long time. 
The parameter space of the configurations of points on a surface, called the Hilbert scheme of points, has been extensively studied. In fact, the birational geometry of this Hilbert scheme has recently had advances coming from Bridgeland stability, derived categories, and interpolation for vector bundles. Using these techniques, questions about the cone of effective divisors or the stable base locus decomposition (SBLD) have satisfactory answers for some interesting cases, e.g. \cite{BM, ABCH13, CHW, LZ1, Nuer2}. However, it remains a difficult problem to determine the effective cone or the SBLD of the Hilbert scheme of points of most surfaces, even in the case of rational surfaces. The present work studies the cone of effective divisors and the SBLD using two approaches: via the augmented base loci and via Severi varieties. We work over the field of complex numbers.

\medskip\noindent
Suppose $X$ is a smooth surface such that the effective cone or the stable base locus decomposition of the Hilbert scheme of points $X^{[n]}$ is known. It is desirable to know how these objects behave on the blowup of $X$ (or the blow down). This is the first topic we investigate. We show that the effective cone as well as the stable base locus decomposition of $N^1(X^{[n]})$ behave according to the weak Lefschetz Principle. Let us be more precise.

\medskip\noindent
Suppose $X$ and $Y$ are smooth surfaces related by a birational morphism $f:X\rightarrow Y$.
We aim to understand how the effective cone and the SBLD of the Hilbert schemes $X^{[n]}$ and $Y^{[n]}$ interact. The morphism $f$ induces a rational contraction, $F:X^{[n]}\dasharrow Y^{[n]}$, in the sense of \cite{HK00}. Most importantly, the pullback $F^*$ induces an embedding of the N\'eron-Severi groups 
$$F^*:N^1(Y^{[n]})\hookrightarrow N^1(X^{[n]}).$$
In this context, the weak Lefschetz Principle says that it should be possible to determine the effective cone and the stable base locus decomposition of $N^1(Y^{[n]})$ by restricting those of $N^1(X^{[n]})$.

\medskip\noindent
In Proposition \ref{EFF}, we show that this is the case for the cone of effective divisors: the cone $\EFF(X^{[n]})$, when restricted to the image of $F^*$, yields the effective cone $\EFF(Y^{[n]})$. However the behavior of the stable base locus is more subtle. For one thing, the image of $F^*$ (inside $N^1(X^{[n]})$) is fully contained in a wall of the stable base locus decomposition. Hence, the restriction of walls provides no information. However, we analyze the base locus of a divisor perturbed by a small multiple of an ample divisor, which is called the augmented base locus. The stable base locus is constant across a wall, but the augmented one is not. It turns out that the stable base locus decomposition of $Y^{[n]}$ can be recovered from the augmented base locus decomposition of $X^{[n]}$ in good cases. Equivalently in such cases, we can say that the stable base locus decomposition of $Y^{[n]}$ gives a slice of the augmented base locus decomposition of $X^{[n]}$. Intuitively,  this is the content of our first main result which we now state. We refer the reader to Theorem \ref{MTheo} for the precise details. 

\begin{thm1} Suppose the morphism $f:X \to Y$ is a series of blow ups at general points. Then, the linear augmented stable base locus decomposition of $\EFF(X^{[n]})$ when restricted to the image $F^*(N^1(Y^{[n]}))$ is equal to the linear stable base locus decomposition of $\EFF(Y^{[n]})$.
\end{thm1}

\medskip \noindent

\medskip\noindent
In order to prove Theorem \ref{MTheo}, we first show that the restriction of the effective cone of divisors of $X^{[n]}$ to $Y^{[n]}$ is the effective cone (Proposition \ref{EFF}). We then use the augmented base loci to further decompose the walls of the SBLD of $X^{[n]}$. If we denote the linear part of the (augmented) stable base locus decomposition of $\EFF(X^{[n]})$ by $\Delta_{X^{[n]}}$ (Definitions \ref{Delta}, \ref{Delta2}), then we may state Theorem \ref{MTheo} as follows $$\Delta_{X^{[n]}}\vert_{V}=F^*\Delta_{Y^{[n]}},$$ where $V=F^*\left(N^1(Y^{[n]})\right)$.

\medskip\noindent
Recently, the weak Lefschetz Principle has been studied in the context of birational geometry \cite{HLW01, Jow11,LHM, Ottem, Oka16}. In our present context, Theorem A says that even though the weak Lefschetz Principle fails for the stable base locus, it holds for the augmented stable base locus. As far as we know, this is the first time this principle has been studied in these terms.

\medskip\noindent
A consequence of the previous result is that if $\EFF(X^{[n]})$ has a finite polyhedral stable base locus decomposition, then the stable base locus decomposition of $\EFF(Y^{[n]})$ is also finite polyhedral, Corollary \ref{imageMDS}. In \cite{Oka16}, Okawa shows that the image of a Mori dream space under a surjective morphism is a Mori dream space. Actually, Okawa shows more: in this case, the weak Lefschetz Principle holds for the strong Mori equivalence, which identifies Mori equivalent line bundles and refines the stable base locus decomposition. In this case, we are looking at a particular case beyond these theorems where the map is a rational contraction rather than a surjective morphism.

\medskip\noindent
We apply Theorem \ref{MTheo} to investigate the behavior of the stable base loci in the case of Hirzebruch surfaces $\mathbb{F}_{r}$. In Proposition \ref{BLDHirzebruch} we describe how walls in $\EFF(\mathbb{F}_{r+1}^{[n]})$ induce walls in $\EFF(\mathbb{F}_{r}^{[n]})$ even though there is no morphism between $\mathbb{F}_{r+1}$ and $\mathbb{F}_{r}$.

\medskip\noindent
While Theorem A may provide information about some extremal rays of $\EFF(X^{[n]})$ given $X$'s minimal models, we also want to investigate effective divisors which are not (necessarily) induced by surfaces birational to $X$. This would allow us to know a bigger region of the effective cone. To do this, let us exhibit effective divisors on $X^{[n]}$ coming from Severi varieties.

\medskip\noindent
Let us consider a line bundle $\mathcal{L}$ on the smooth surface $X$ such that the linear system $|\mathcal{L}|$ is not empty. Let $n$ be an integer $0\le n \le p_a(C)$, where $C\in |\mathcal{L}|$ and $p_a$ stands for the arithmetic genus. Let $V_{n}(\mathcal{L})$ be the Severi variety, which generically parametrizes irreducible curves in $|\mathcal{L}|$ with exactly $n$ nodes and no other singularities. This variety $V_{n}(\mathcal{L})$ is a locally closed subscheme of the projective space $|\mathcal{L}|$, and we will study its image under the rational map $f : V_{n}(\mathcal{L}) \rightarrow X^{[n]}$ which sends a curve to the scheme supported at its nodes.

\begin{defin}
Let $f$ be the forgetful map $f:V_{n}(\mathcal{L})\rightarrow X^{[n]}$, which sends a curve to its nodes. We define the \textit{Severi locus}, $\mathrm{Sev}(n,\mathcal{L})$, as the closure of the image of the forgetful map $f$, $$\mathrm{Sev}(n,\mathcal{L})=\overline{\mbox{Im}(f)}.$$ When this locus is not empty and has codimension $1$, we call it a \textit{Severi divisor}.
\end{defin}

\medskip\noindent
Whenever we have a Severi divisor, our second main result Theorem B computes its divisor class in $\mathrm{N}^1(X^{[n]})$. Let us state this result and refer the reader to Theorem \ref{SEVgen} for the precise details. We need the following notation: if $D\subset X$ is a reduced effective divisor, then $D[n]\subset X^{[n]}$ denotes the divisor which generically parametrizes subschemes whose support intersects $D$. 

\begin{thm2}\label{SEVgen}
Let $C\subset X$ be a curve contained in a smooth projective surface with $h^1(\mathcal{O}_X)=0$. Assume the Severi variety $V_{n}(|C|)$ has the expected dimension, generically parametrizes irreducible curves with $n$ nodes, and the class $K_X+3C$ is effective. Then the class of the Severi divisor is  \[ \mathrm{Sev}(n,|C|)=(K_X+3C)[n]-\tfrac{5}{2}B[n],\] as long as $\mbox{dim }|C|=3n-1$, or it is empty. 
\end{thm2}

\medskip\noindent
Note if the Severi variety is reducible, it generically parametrizes irreducible curves means that statement is true for the generic point of each component.
As a consequence of the previous result, and extending work of Arbarello and Cornalba \cite{AC81}, we get that the forgetful map $f$ is finite.

\begin{cor3}
Under the assumptions of Theorem B, the forgetful map, which sends a nodal curve to the subscheme supported at its nodes, $f:V_{n}(|C|)\rightarrow X^{[n]},$ is finite.
\end{cor3}

\medskip\noindent
We will compute the classes of Severi divisors, first in the case of the plane $\PP^2$ and then on any regular surface. This will demonstrate that these divisors realize walls in the stable base locus decomposition of the effective cone of $\PP^{2[n]}$ in many cases. In such cases, the base loci of Severi divisors are configurations of point that fail to impose independent conditions on $|\mathcal{O}_{\PP^2}(d)|$, for certain $d$, Corollary \ref{IMPOSING}. We also observe that Severi divisors provide some examples of divisor classes which were not known to be effective for Hirzebruch surfaces.

\subsection*{Related work}
An important part of \cite{Hui, CHW} is the explicit description of the cone of effective divisors of $\PP^{2[n]}$ (they proved more, but let us discuss only the Hilbert scheme). These papers show that the (interesting) extremal ray of $\EFF(\PP^{2[n]})$ is generated by a Brill-Noether divisor. That is, configurations of points such that there exists a section of a suitable vector bundle vanishing on them. A technical part of the papers is to show that a generic point in $\PP^{2[n]}$ is not of this type; this is called interpolation. In this paper, instead of dealing with the interpolation of higher rank vector bundles, we study configurations of points such that there exists a section of a suitable line bundle vanishing on them to second order. In other words, we focus on the family of points where sections of line bundles vanish and their first derivatives vanish as well. This yields Severi divisors.

\medskip\noindent
In the case of $X=\mathbb{P}^2$, the walls of the stable base locus decomposition correspond to walls in the Bridgeland stability manifold \cite{ABCH13, BM14, LZ1, LZ2}. If the correspondence of the Bridgeland walls with the stable base locus walls holds for a minimal surface, then our results may help understanding such a correspondence for non-minimal surfaces, and moreover they suggest a specific structure of the stability manifold. Indeed, it would be interesting to see if the stability manifold $\mathrm{Stab}(S)$ (or some slices of it), for a certain surface $S$, sits as a complex submanifold of $\mathrm{Stab}(Bl_p(S))$, and furthermore, to verify whether the Bridgeland chamber decomposition on the submanifold is induced from the ambient one, following what Theorem \ref{MTheo} suggests. 

\subsection*{Organization of the paper}
Section \ref{sec: prelim} contains some preliminaries on Hilbert schemes of points on surfaces. Section \ref{SEC3} contains definitions of (augmented) stable base locus and the proof of Theorem \ref{MTheo}. We have also included here the discussion on walls in the SBLD in the case of Hirzebruch surfaces. Section \ref{SEC4} contains the discussion about Severi divisors, in particular the proof of Theorem \ref{SEVgen}. We finish the paper with examples of these divisors on Hirzebruch surfaces and $K3$ surfaces. We included an example of a Severi divisor in $\mathbb{F}_1^{[12]}$ that does not come from $\mathbb{P}^{2[12]}$.

\section*{acknowledgments} \noindent
We would like to thank Gabriel Bujokas, Izzet Coskun, Joe Harris, Rob Lazarsfeld, Cristian Mart\'inez, Alex Massarenti, Benjamin Schmidt, and Edoardo Sernesi for useful conversations about this project. Special thanks to Edoardo Sernesi for sharing with us his insights in the proof of Theorem \ref{SEVgen}. The first author is also grateful to the Dipartimento di Matematica e Fisica in Roma Tre for providing ideal working conditions where part of this work was done. We also thank the organizers of the conference II CNGA at CMO where parts of this work were carried out. During the preparation of this article the first author was partly supported by the CONACYT grant CB-2015/253061, and the second author was partially supported by the NSF grant DMS-1547145.

\medskip
\section{Preliminaries on the Hilbert Scheme of points on surfaces}
\label{sec: prelim}

\noindent
Let us begin by recalling some basic facts about Hilbert schemes of points.
For a more complete introduction to the subject see work of Nakajima \cite{Na1}.

\medskip\noindent
Let $Y$ be a smooth projective surface with $h^1(\mathcal{O}_Y)=0$. Let us denote by $\{D_1,\ldots, D_k\}$ a set of generators of the Picard group $\Pic(Y)$. 
For $n>0$, recall that the set of unordered $n$-tuples of points of $Y$ is called the $n$-th symmetric product and is denoted by $Y^{(n)}$. 
This space is the quotient of the $n$-th product $Y^n$ by the symmetric group $\mathfrak{S}_{n}$, where the action is permutation of the coordinates. 
The symmetric product is singular along the locus of tuples with a repeated point of $Y$, which naturally leads to a desire for a better moduli space of points; such a space is the Hilbert scheme.

\medskip\noindent
Observing that a subscheme $\Gamma\subset Y$ which consists of $n$ distinct points has Hilbert polynomial $n$, it induces a point of the appropriate Hilbert scheme.
Motivated by this, let $Y^{[n]}$ denote the Hilbert scheme which parametrizes subschemes of $Y$ with constant Hilbert polynomial $n$. 

\medskip\noindent
The first step in understanding the birational geometry of $Y^{[n]}$ is to understand the divisors on it.
In order to state our results, we first define a Weil divisor on $Y^{[n]}$ given a reduced divisor on $Y$.
Define $D[n]$ to be the divisor in $Y^{[n]}$ of subschemes whose support intersects a general fixed reduced curve with class $D$. 
When it is clear, we will simply write $D$ in place of $D[n]$.

\begin{thm}(Fogarty \cite{FG1})
The Hilbert scheme $Y^{[n]}$ is a smooth irreducible projective variety of dimension $2n$. 
This space $Y^{[n]}$ admits a natural morphism to the symmetric product $Y^{(n)}$ called Hilbert-Chow morphism $$h:Y^{[n]}\longrightarrow Y^{(n)}.$$
The morphism $h$ is birational and gives a crepant desingularization of $Y^{(n)}$.
Furthermore, if $h^1(\mathcal{O}_Y) =0$, then
$N^1(Y^{[n]})=\text{Pic}(Y^{[n]})\otimes \mathbb{Q}$ is spanned by $D_1[n]$, $\ldots$, $D_k[n]$, and $B$ where $B$ is the exceptional divisor of $h$.
\end{thm}

\medskip\noindent
This result implies that the Weil divisors on $Y^{[n]}$ are also Cartier divisors. Thus, the Weil divisors $D_i[n]$ and $B[n]$ suffice to generate all divisors of $Y^{[n]}$ (over $\mathbb{Q}$).
For example, in case $Y=\PP^2$, let $H$ be the locus of subschemes $\Gamma \in \PP^{2[n]}$ such that $\Gamma\cap L\neq 0$, where $L\subset \PP^2$ is a fixed line. 
Also, let $B$ be the locus of non-reduced subschemes of $\PP^2$ of dimension zero and length $n$. 
Alternatively, we can describe $H$ and $B$ in terms of the Hilbert-Chow morphism: $H:=h^*\mathcal{O}(1)$ is the pullback of the ample generator of $\text{Pic}\left(\mathbb{P}^{2(n)} \right)$ and $B=Exc(h)$ is the exceptional divisor.

\medskip\noindent
Analogously to the divisor $D[n]$, define the curve $C_D[n]$ in $Y^{[n]}$ by fixing $n-1$ general points of $Y$ and then varying an $n$-th point along a fixed curve of class $D$.
For $D_i$, we will abuse notation and write $C_i[n]$. By Fogarty's theorem, we know that the space of $1$-cycles on $Y^{[n]}$ is generated by
$$N_1(Y^{[n]}) = \left<C_0[n], C_1[n],\cdots, C_k[n] \right>$$
where  $C_0[n]$ is the curve defined by fixing $n-2$ general points of $Y$, a general point of a fixed curve $D_0$, and then varying an $n$-th point along  that curve. The birational invariant we will study first is defined as the closure of the cone of divisors classes which are effective; it is denoted by $\EFF(X^{[n]})\subset N^1(X^{[n]})$.


\noindent

\medskip
\section{Stable base locus decomposition of the cone of effective divisors}\label{SEC3}
\medskip\noindent
Throughout this section we make use of the following notation. Let $f:X\rightarrow Y$ be a birational morphism between smooth surfaces with $h^1(\mathcal{O}_Y) = 0$. Then, there is an induced rational map at the level of Hilbert schemes, $F:X^{[n]}\dasharrow Y^{[n]}$.
Our first result claims that via $F$, we can recover the effective divisors of $Y^{[n]}$ from those of $X^{[n]}$. We recall from \cite{HK00} the following definition.

\begin{defn} Let $F:X\dasharrow Y$ be a dominant rational map, where $Y$ is normal and projective. We say that $F$ is a \textit{rational contraction} if there exists a resolution of $F$
\[\begin{diagram}\label{RES}
&& W && \\
& \ldTo^{q} & & \rdTo^{\tilde{F}} &\\
X &&\rDashto^{F}&& Y,
\end{diagram}\]
where $W$ is smooth projective, $q$ is birational, and for every $q$-exceptional effective divisor $E$ on $W$, we have that $$\tilde{F}_*(\mathcal{O}_W(E))=\mathcal{O}_Y.$$
\end{defn}

\noindent\medskip
With the notation as in the previous definition, this Lemma will be used in Proposition \ref{EFF}.
\medskip
\begin{lem}\label{CONT}
Let $F:X\dashrightarrow Y$ be a birational map between irreducible normal $\mathbb{Q}$-factorial varieties. Denote by $W$ a fixed resolution of $F$, then the following are equivalent,
\begin{enumerate}
\item The map $F$ is a rational contraction.
\item Any $E\subset W$ prime divisor $q$-exceptional is also $\tilde{F}$-exceptional.
\item There are two isomorphic open sets $U\subset X$ and $V\subset Y$, such that $\text{codim} (Y\backslash V)\ge 2$.
\end{enumerate}
\end{lem}
\begin{proof} Let us show (1) implies (2) using proof by contrapositive. Let $W$ be a resolution in the sense of the previous definition. Let $E$ be a prime $q$-exceptional divisor and suppose it is not $\tilde{F}$-exceptional. Then $\tilde{F}(E)$ is a non-trivial divisor which satisfies $\tilde{F}_*(\mathcal{O}_W(E))\ne \mathcal{O}_Y.$ Indeed, let us write $N=Y\backslash \tilde{F}_{|E}(exc(\tilde{F}))$, and notice that $\mathrm{codim }(Y\backslash N)\ge 2$. Hence, $\tilde{F}(E)_{|N}$ is a prime Cartier divisor which is not trivial.
Therefore, $\tilde{F}_*(\mathcal{O}_W(E))_{|N}=\mathcal{O}_N(\tilde{F}(E))\ne \mathcal{O}_N$. It follows that $$\tilde{F}_*(\mathcal{O}_W(E))\neq\mathcal{O}_{Y},$$
which means that $F$ fails to be a rational contraction. Now the contrapositive yields the claim.

To see that (2) implies (1), observe that if $E$ is $\tilde{F}$-exceptional, then $\tilde{F}_*\mathcal{O}_W(E)=\mathcal{O}_Y$.

Observe that the item $(2)$ holds if and only if the item (3). Indeed, let $U\subset X$ be the open set over which $F$ is injective. Note the codimension of the complement of $F(U)$ is bigger or equal than $2$ if and only if any $q$-exceptional divisor is also $\tilde{F}$-exceptional. This completes the proof.
\end{proof}

\medskip
\begin{prop}\label{EFF}
Let $f:X\rightarrow Y$ be a birational morphism between smooth surfaces. If we denote by $F:X^{[n]}\dasharrow Y^{[n]}$ the induced map between Hilbert schemes, then $F$ is a rational contraction and $$\EFF\left(X^{[n]}\right)\vert_{V}=F^*\left(\EFF\left(Y^{[n]}\right)\right),$$
where $V$ stands for the image $V=F^*\left(N^1\left(Y^{[n]}\right)\right)$.
\end{prop}

\begin{proof} Since a birational morphism between smooth surfaces is a composition of blow downs, it suffices to analyze the case $X=\mathrm{Bl}_pY$, $i.e.$, the blowup of a point $p\in Y$.

Observe that $F$ is birational and that there are open sets $U\subset X^{[n]}$ and $T\subset Y^{[n]}$, such that $F:U\rightarrow T$ is an isomorphism with $\mathrm{codim }(Y^{[n]}\backslash T)\ge 2$.  Indeed, the open set $U$ is the complement of the divisor $E[n]$, which generically parametrizes sub schemes whose support intersects the exceptional divisor $E\subset X$. The open set $T$ is complement of the family of subschemes whose support contains the point $p\in Y$. It follows from Lemma \ref{CONT} that the map $F$ is a rational contraction.

Let us now address the claim about the effective cone. This is a claim about the equality of two sets, so we are going to argue by showing two inclusions.

For the first inclusion, let us start by considering an effective divisor $D \in \text{Eff}(Y^{[n]})$. 
Then, $F^*(D)$ is again an effective divisor whose class is in $F^*(N^1(Y^{[n]}))$.
Thus, 
\[ F^* \left( \text{Eff}\left(Y^{[n]}\right)\right) \subset \text{Eff}\left(X^{[n]}\right) \cap F^*\left(N^1\left(Y^{[n]}\right)\right). \] 
Taking closures, we have that
\[F^*\left(\EFF(Y^{[n]})\right) \subset \EFF\left(X^{[n]}\right) \cap F^*\left(N^1\left(Y^{[n]}\right)\right). \] 

For the opposite inclusion, let us consider $D' \in  \text{Eff}\left(X^{[n]}\right) \cap F^*\left(N^1\left(Y^{[n]}\right)\right)$, which means that $D'$ is an effective divisor and $D '= F^*(D)$, for a divisor class $D\in N^1(Y^{[n]})$.
Assume $D$ is not pseudo-effective. Then, there exists a moving curve class $C\in N_1(Y^{[n]})$, such that $C \cdot D <0$. Observe there exists a moving curve class $C' \in N_1( X^{[n]})$ defined as the class of the inverse image of a general element in $C$. Note, $F_*(C')=C$. We want to show that $C' \cdot D'<0$. To do this, let $W$ be a resolution of the rational map $F$ in the sense of Definition \ref{RES} ,
\[\begin{diagram}
&& W && \\
& \ldTo^{q} & & \rdTo^{\tilde{F}} &\\
X^{[n]} &&\rDashto^{F}&& Y^{[n]},
\end{diagram}\] 
and observe that $q^*(D')=\tilde{F}^*(D)$. Indeed, the maps $q$ and $\tilde{F}$ are birational morphisms, hence the induced pullbacks, $q^*$ and $\tilde{F}^*$, are injective. Furthermore, there is a curve class $C_0\in N_1(W)$ such that $q_*(C_0)=C'$ and $\tilde{F}_*(C_0)=C$. Observe that $\tilde{F}^*(D). C_0<0$ by the projection formula. The projection formula again now implies that
\begin{equation*}
    \begin{aligned}
    q_*(\tilde{F}^*(D).C_0)&=q_*(q^*(D').C_0)\\
    &=D'.C'<0,
\end{aligned}
\end{equation*}
which is a contradiction due to the fact that $D'$ is assumed pseudo-effective.
Therefore, the divisor class $D$ must be pseudo-effective and we get
\[\EFF\left(X^{[n]}\right) \cap F^*\left(N^1\left(Y^{[n]}\right)\right) \subset F^*\left(\EFF(Y^{[n]})\right),\] as desired. This completes the proof. 
\end{proof}

\medskip\noindent
We have compared in the previous proposition the effective cones $\EFF(Y^{[n]})$ and $\EFF(X^{[n]})$ via the rational contraction $F$. We now want to compare the stable base locus decomposition of both effective cones. This is the content of Theorem \ref{MTheo} and Corollary \ref{cor: SBLD V}. In order to prove these results, let us recall the notions involved.

\begin{defn}
Let $D$ be a divisor on a smooth projective variety $X$.
The \textit{base locus} of $D$, denoted $\mathbf{Bs}(D)$, is the intersection of all divisors $D'$ linearly equivalent to $D$, \[\mathbf{Bs}(D):= \bigcap_{D' \in \vert D \vert}D'.\] 
\end{defn}

\medskip\noindent
The base locus is a natural construction, but varies unexpectedly for divisors with similar equivalence classes. This situation leads to the following definition.
\begin{defn}
The \textit{stable base locus} of $D$, denoted $\mathbf{B}(D)$, is the intersection of the base locus of all positive multiples of $D$, \[\mathbf{B}(D):= \bigcap_{m>0}\mathbf{Bs}(mD).\]
\end{defn}

\medskip\noindent
Equivalently, this is the intersection of all divisors $D'$ linearly equivalent to some multiple of $D$, \[\mathbf{B}(D)= \bigcap_{D' \in \vert mD \vert,m>0}D'.\]

\medskip\noindent
On a Mori dream space $X$, there is a open set of the space $N^1(X)$ on which the stable base locus is well defined and locally constant \cite{ELMNP}. 
We call this the \textit{stable locus}. 
The complement of the stable locus are the \textit{walls} of the stable base locus decomposition (SBLD) and they are defined by linear equations. We will denote the collection of such walls by $\Delta_X$. Some of these properties hold in more generality which we will discuss two paragraphs below. The techniques we use give us control over linear walls, which leads us to make the following definition.



\begin{defn}\label{MORIsurface}
Let $X$ be a smooth surface with $h^1(\mathcal{O}_X)=0$. We call $X$ a \textit{Mori surface} 
if the Hilbert scheme $X^{[n]}$ has a linear stable base locus decomposition for all $n$. In other words, the walls of the SBLD, denoted $\Delta_X$, are defined by linear equations for all $n$. 
\end{defn}


\medskip\noindent
Consider a birational morphism between surfaces $f: X\rightarrow Y$, such that $X^{[n]}$ is a Mori dream space. We aim to show that the walls of the SBLD of $N^1(Y^{[n]})$ are defined by linear equations. Furthermore, they are determined by divisors on $X^{[n]}$ (Theorem \ref{MTheo}). For example, if $X$ is a Fano surface, then both $X^{[n]}$ and $Y^{[n]}$ are Mori dream spaces \cite{BC13}, and we may try to restrict the structure of the SBLD of $N^1(X^{[n]})$ to $N^1(Y^{[n]})$. However, when we attempt to do this we run into a problem: if $E$ is the exceptional divisor of $f$, the curve $C_E[n]$ defines a wall in $N^1\left( X^{[n]} \right)$ which completely contains $F^*\left(\text{Eff}(Y^{[n]})\right)$. Hence, the restriction of the stable base locus of $N^1(X^{[n]})$ gives no information about the SBLD of $N^1(Y^{[n]})$.

\medskip\noindent
In order to deal with this, we recall the definition of the augmented (restricted) stable base loci.
\begin{defn}\label{def:aug}
The \textit{augmented (resp., restricted) stable base locus} of $D$, denoted by $\mathbf{B}_+(D)$ (resp., $\mathbf{B}_-(D)$), is the stable base locus of $D-\epsilon A$ (resp, $D+ \epsilon A)$, that is
\[\mathbf{B}_+(D) := \mathbf{B}(D-\epsilon A) \quad\left(resp,\  \mathbf{B}_-(D) = \mathbf{B}(D+\epsilon A) \right),\]
for any ample divisor $A$ and $0<\epsilon<<1$.
These are independent of the choice of $A$.
\end{defn}

\medskip\noindent
The augmented and restricted stable base loci will allow us to detect the SBLD of $N^1(Y^{[n]})$ inside the wall induced by the curve $C_E[n]$ inside $N^1(X^{[n]})$.
Every numerical class has a well defined augmented and restricted base locus which are invariant under scaling. 
The stable locus is precisely where the augmented base locus is equal to the restricted base locus. Moreover, it was shown in \cite[Prop. 1.26]{ELMNP}, that the locally constant property discussed before Definition \ref{MORIsurface} is satisfied by divisor classes such that $B_+(D)=B_{-}(D)$.
We will primarily be interested in the augmented stable base locus, so we will state a decomposition with respect to that alone.

\medskip\noindent
\begin{defn}\label{Delta} Let $X$ be a smooth projective variety.
The \textit{augmented stable base locus decomposition, ASBLD} of $N^1(X)$, is the partition of $\overline{\mathrm{Eff}}(X)$ such that  $\mathbf{Bs}_+(D)$ is fixed for every class $D$ in a fixed element of the partition. A \textit{wall} of the ASBLD is the interior of the boundary of any element in the partition given by the ASBLD. 
\end{defn}

\medskip\noindent
We will use the notion of ASBLD in order to define the ASBLD of a subspace $V$ of the N\'eron-Severi space. This will allow us to compare the ASBLD similarly to Proposition \ref{EFF}, where we compared the effective cones. Observe that the chambers of stable base locus decomposition of $X$ are the interiors of the chambers in the ASBLD of $V$.
Similarly, the curves defining the walls of the SBLD fully determine the ASBLD and vice versa. 


\begin{defn}\label{Delta2} Let $X$ be a smooth projective variety and $V$ a subspace of $N^1(X)$.
The \textit{augmented stable base locus decomposition} of $V$, denoted $\Delta_X\vert_V$ is the restriction of the ASBLD of $X$ to $V \cap \overline{\mathrm{Eff}}(X)$.
A \textit{wall} of the ASBLD is again the interior of the boundary of any element in the partition given by the ASBLD.
\end{defn}

\medskip\noindent
This definition formalizes the restrictions to $V$ of the walls in the SBLD of $N^1(X)$ that do not contain $V$.
In a Mori dream space, each wall in the SBLD is defined by a curve class $C$ which is dual to every divisor along the wall.
The subvariety covered by the curves in the class $[C]$ is contained in the base locus on one side of the wall, but often not on the other side.
This is equivalent to having $\dim(V)-1$ linearly independent divisors on the wall whose augmented base locus agree and are not equal to their restricted base locus, which also all agree.

\medskip\noindent
In this paper, every element of the partition of the N\'eron-Severi space (resp., subspace $V$) has a (resp., relative) full dimensional interior and any pair of chambers whose closures intersect along a codimension 1 subspace are separated by a wall. Informally, if the decomposition is finite polyhedral, then one can apply the definition to each subspace containing a wall in order to further decompose such a wall.

\medskip\noindent
We want to apply these notions to the Hilbert scheme of points $Y^{[n]}$ as we vary the surface $Y$ within its birational class. Indeed, let $f:X\rightarrow Y$ denote the blowup of $Y$ at a general point, and $E$ the exceptional divisor. 
Let $V=F^*\left(N^1\left( Y^{[n]}\right)\right)$ be a subspace of $N^1\left(X^{[n]}\right)$.
Our goal for the rest of this section is to show the restriction of the ASBLD to $V$ is equal to the pull back of the SBLD.
We first show one direction of this statement. Let us denote by $E\subset X$ the exceptional divisor and the upper half space  which consists of divisor classes for which $E[n]$ has a positive coefficient by $\mathbb{H}_{E}\subset N^1(X^{[n]})$.

\begin{lem}\label{Prop10}
Any wall of the SBLD of $X^{[n]}$ in $\mathbb{H}_{E}$ induces a wall in the SBLD of $Y^{[n]}$.
Since the ample cone of $X^{[n]}$ lies in the other half space, this means every wall of the ASBLD of $V$ induces a wall of the SBLD of $Y^{[n]}$.
\end{lem}

\begin{proof} 
Let $D$ be a divisor on $Y^{[n]}$, $F^*(D) = D'$, and $U$ denote the open set where the map $f:X\rightarrow Y$ is an isomorphism. Note in case we have a curve class $C\in N_1(X)$, then we may consider the hyperplane in $N^1(X)$ defined by it: the divisor classes which pair zero with $C$.

The proof proceeds in three steps.
First, we show any difference in base locus between chambers in $\mathbb{H}_{E}$ differ by a locus in $U^{[n]}$.
Second, we show that the curve defining the hyperplane in $N^1(X^{[n]})$ (which restricts to an augmented wall in $V$), defines the same hyperplane in $N^1(Y^{[n]})$ (which gives rise to a wall in the SBLD).
Finally, we show the latter hyperplanes define walls by exhibiting enough divisors whose restricted and augmented base loci are different.

For the first step, let $C_E$ be the curve defined as those schemes containing $n-1$ general fixed points and whose $n$-th point is on $E$.
Observe that $E[n] C_{E} <0$, which implies that all representatives of $[C_E]$ are contained in $E[n]$. If $D'\cdot C_E =0$, then it follows that $E[n]$ is in the (augmented) base locus for every divisor of the form $D'+aE[n]$, where $a\in \mathbb{Q}$ is positive.
Thus, any wall of the SBLD in $\mathbb{H}_{E}$ separates two cones of divisors whose base loci differs only by schemes lying in $U^{[n]}$ (the complement of $E[n]$). Consequently, such base loci can be considered as points in either $X^{[n]}$ or $Y^{[n]}$.

Let $C'$ be a curve defining one such wall in $X^{[n]}$.
Since representatives of the class $[C']$ cover some portion of the difference of the base loci ($i.e.$, points in $U^{[n]}$), we can assume this curve $C'$ has an open subset contained in $U^{[n]}$. Observe we may take the closure of such an open set either in $X^{[n]}$ or $Y^{[n]}$ (as $U^{[n]}$ is an open set in both spaces).
We will denote by $C'$ the closure in $X^{[n]}$ and by $C$ the closure in $Y^{[n]}$.
As the blowup point $X=Bl_pY$ was general, then it is not in the support of any subscheme parameterized by some curve equivalent to $C$. This means that we can assume without loss of generality that the intersection of $C$ with a divisor $D$ occurs in $U^{[n]}$. Likewise the intersection $C'. D'$ occurs in $U^{[n]}$. 
Then by the push-pull formula applied over $U^{[n]}$, we have that
\begin{equation*}\label{projection}
\begin{aligned}
C \cdot D = F_*(C').F_*(D')&=F_*(C'.F^*F_*(D'))\\
&=F_*(C'.D').
\end{aligned}
\end{equation*}

In particular, if $C$ defines a hyperplane in $N^1(Y^{[n]})$, then this hyperplane is the same as the restriction to $V$ of the hyperplane defined by $C'$ on $N^1(X^{[n]})$, and vice versa.

Finally for the last step, it suffices to show that the curve $C'$ defining an wall in $\mathbb{H}_{E}$ does in fact define a wall in $Y^{[n]}$ via $C$. Recall $C = F_*(C')$, which means we map the open set which is isomorphic and take the closure.
We know that $C'$ defines a wall of the ASBLD of $F^*(N^1(Y^{[n]}))$.
In particular, there are $\rho(Y)$ linearly independent divisors $F^*(D)$ such that $C' \cdot F^*(D) =0$ and each of those divisors has a sequence $\{ F^*(D_i)\}$ of divisors such that augmented base locus of $F^*(D_i)$ is different from that of $F^*(D)$ and are all equal.
By the reasoning above the difference between their augmented base loci and that of $F^*(D)$ must consist entirely of points in $U^{[n]}$.
Then the augmented base locus of a divisor on $Y^{[n]}$ and the augmented base locus of its pullback can only differ by points outside of $U^{[n]}$ so the augmented base locus of each $D_i$ is different from that of $D$ in the same way.
Thus, $C$ defines a wall as desired.
\end{proof}

\medskip\noindent
The converse of the previous lemma is the main result of this section and it is what we show next.

\begin{thm}
\label{MTheo} Let $f:X\rightarrow Y$ a birational morphism between Mori surfaces which is a sequence of blow ups of general points. Then, the restriction of the ASBLD to $V$ is equal to the pull back of the SBLD via the rational map $F:X^{[n]}\dasharrow Y^{[n]}$. That is to say
$$\Delta_{X^{[n]}}\vert_V = F^*(\Delta_{Y^{[n]}}),$$
where $V=F^*N^1(Y^{[n]})$.
\end{thm}

\begin{proof}
We may assume that $X$ is a blowup $X=Bl_yY$, and that $f:X\rightarrow Y$ is the blowup map. It follows from Lemma \ref{Prop10} that a wall of the ASBLD of $V$ induces a wall of the SBLD of $Y^{[n]}$. In other words, $ F^*(\Delta_{Y^{[n]}})\supset\Delta_{X^{[n]}}\vert_V $. Then, it suffices to show that given a wall of $Y^{[n]}$ ($i.e.$, a wall in $F^*(\Delta_{Y^{[n]}})$) induces a wall of $V$.

Let $C$ be a curve class defining a wall of the SBLD of $Y^{[n]}$ and $D_i$ the $\rho(Y)-1$ linearly independent divisor class on that wall ($i.e.$, $C\cdot D_i =0$ in $Y^{[n]}$).
Since the point $y\in Y$ is a generic point and the support of the elements of $C$ sweeps out a curve, then $y$ is not part of the support of any subscheme contained in the general element of $C$.
Therefore, we can consider $C$ as a curve class $C'$ on $X^{[n]}$. As the $F^*(D_i)$ are still linearly independent, it suffices to show each $F^*(D_i)$ is dual to $C'$ and the augmented stable base locus of $F^*(D_{i,j})$ differs from that of $F^*(D_i)$.
Since $C' \cdot F^*(D_i)\vert_E = C \cdot D_i\vert_y=0$ as $y$ is general, we can apply the projection formula to each curve's open set in $U^{[n]}$. 
Thus, by the projection formula, \[C' \cdot F^*(D_i) = C \cdot D = 0,\] so $F^*(D_i)$ is dual to $C'$.

Since the augmented stable base locus of $F^*(D_i)$ and $D_i$ differ only by points not in $U^{[n]}$ for all divisors $D_i$ and the augmented stable base locus of $D_{i,j}$ differs from that of $D_i$ by points in $U^{[n]}$, the augmented stable base locus of $F^*(D_{i,j})$ differs from that of $F^*(D_i)$ as desired.
\end{proof}

\medskip\noindent
Relaxing the assumption that $y$ was a general point, a slight modification of the same argument gives the following corollary.

\begin{cor}
\label{cor: SBLD V}
Let $X = \text{Bl}_p(Y)$ where $p$ may no longer be general, and $X$ and $Y$ are Mori surfaces. Then the walls of $\Delta_{X^{[n]}}\vert_V$ differ from those of $F^*(\Delta_{Y^{[n]}})$ only by walls where the base locus on either side differs only by schemes whose support intersects the exceptional divisor or the blown down point.

\end{cor}

\medskip\noindent
As mentioned at the beginning of the section, we also get the previous result for any birational morphism of surfaces.
\begin{cor}
\label{cor: birational SBLD V}
Let $f: X \to Y$ be a birational morphism between Mori surfaces. Then the walls of the ASBLD of $V$ differs from those of $Y^{[n]}$ only by walls where the base locus on either side differs only by schemes whose support intersects the exceptional divisors or the blown down points.
\end{cor}

\medskip\noindent
We now have an immediate corollary that concerns the image of Hilbert schemes of Mori surfaces under rational contractions.

\begin{cor}\label{imageMDS}
Let $f:X\rightarrow Y$ be a birational morphism between surfaces. If $X^{[n]}$ has a finite polyhedral stable base locus decomposition, then so does $Y^{[n]}$.
\end{cor}

\medskip\noindent
Observe that in case $Y^{[n]}$ has infinitely many chambers in its SBLD, then $X^{[n]}$ will as well. In particular, if $Y^{[n]}$ fails to be a Mori dream space due to the presence of infinitely chambers in the SBLD, then $X^{[n]}$ will fail to be a MDS for the same reason.


\begin{ex}\label{DELPEZZO}(SBLD correspondence for Del Pezzo surfaces.)
Let $S_d$ be the degree $9-d$ Del Pezzo surface for $8\geq d \geq  1$. 
Recall that $S_d$ is isomorphic to the blowup of $\mathbb{P}^2$ at $d$ general points. Hence, the Picard group $\Pic(S_d)$ is generated by $H$ and $E_1,\ldots,E_d$, where $H$ is the pull back from $\PP^2$ of a general line and the $E_i$ each denote an exceptional divisor.
Denote the blow up map by $f_d:S_d \to \mathbb{P}^2$ with corresponding birational maps $F_d: S_d^{[n]} \dashrightarrow \mathbb{P}^{2[n]}$. It follows from Theorem \ref{MTheo}, applied to each successive blow up $S_d \to S_{d-1}$, that

$$\Delta_{S_d^{[n]}}\vert_V = F_d^*(\Delta_{\mathbb{P}^{2[n]}}),$$ where $V = F_d^*\left( N^1\left( \mathbb{P}^{2[n]}\right)\right) = \left<H[n],B[n]\right>$.
\end{ex}

\medskip\noindent
\textbf{Example:} 
The following picture exemplifies how the SBLD of $\PP^{2[3]}$ can be interpreted inside the SBLD of $\FF^{[3]}_1 = S_1^{[3]}$ in the subspace generated by $\langle H, B\rangle$. In the case of $\FF^{[3]}_1$, we are drawing a cross-section of the cones. Also, we have shaded the moving cones to draw attention to how they correspond. Note $X_{i,0} = iH-\frac{1}{2}B$ on $\mathbb{F}_1^{[3]}$ and $X_{i} = iH-\frac{1}{2}B$ on $\mathbb{P}^{2[3]}$. 
\begin{center}
\begin{tikzpicture}[scale=.5]
\draw[fill=lightgray] (0,5/3)--(5/2,5/2)--(5/3,0)--(0,0)--(0,5/3);
\draw (-5,0)--(0,5)--(5,0)--(0,-5)--(-5,0);
\draw[thick] (0,5)--(0,-5);
\draw (5,0)--(-5,0);
\draw (2.5,2.5)--(-5,0);
\draw (2.5,2.5)--(-2.5,-2.5);
\draw (2.5,2.5)--(0,-5);
\draw (0,5/3)--(5,0);
\draw (0,5/3)--(10/3,-5/3);
\draw[above left] (0,5) node {$B$};
\draw[above left] (0,5/3) node {$H$};
\draw[above right] (2.5,2.5) node {$F$};
\draw[above left] (0,0) node {$X_{2,0}$};
\draw[below left] (0,-5) node {$X_{1,0}$};

\draw[fill=lightgray] (15,-4)--(12.75,3.5)--(15,4)--(15,-4);
\draw[->] (15,-4)--(20,-4);
\draw[->] (15,-4)--(15,4);
\draw[->] (15,-4)--(12.75,3.5);
\draw[->] (15,-4)--(11.5,2.5);
\draw[above] (20,-3.9) node {$B$};
\draw[above] (15,4.1) node {$H$};
\draw[above left] (12.75,3.5) node {$X_{2}$};
\draw[above left] (11.5,2.5) node {$X_{1}$};

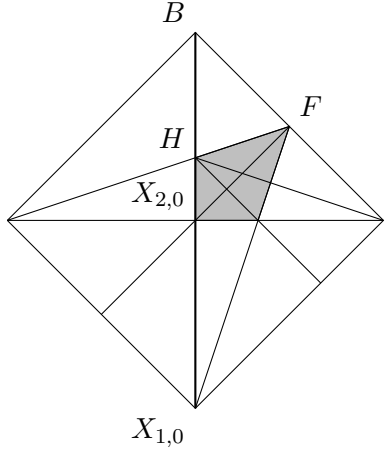
\draw (0,-8) node {The SBLD of $\FF_1^{[3]}$};
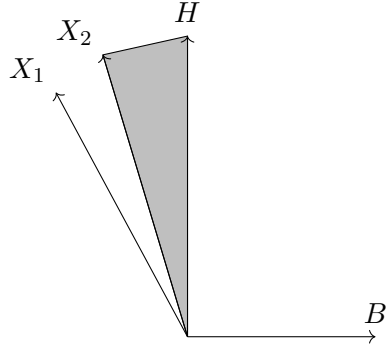
\draw (15,-8) node {The SBLD of $\PP^{2[3]}$};
\end{tikzpicture}
\end{center}

\medskip
\subsection{Chambers of $\mathbb{F}_{r+1}^{[n]}$ induce chambers in $\mathbb{F}_r^{[n]}$}
\label{sec: SBLD correspondence for Hir}

\medskip\noindent
Let $\FF_r$ denote the Hirzebruch surface defined as $\FF_r=\PP(\mathcal{O}_{\PP^1}\oplus \mathcal{O}_{\PP^1}(r))$, with $r>0$.
Denote by $E$ the class in $\Pic(\FF_r)$ of the unique curve of self-intersection $-r$, denote by $F$ the class of a fiber of the ruling of $\FF_r$, and denote by $H$ the class of the strict transform of a general line under the birational map $\FF_r \dashrightarrow \PP^2$.
Note that $H = E +rF$. 

\medskip\noindent
We aim to compare the SBLD of the effective cone of  $\mathbb{F}_{r+1}^{[n]}$ with the SBLD of $\FF_r^{[n]}$.
This result does not follow immediately from Theorem \ref{MTheo} as the birational maps involved are now compositions of a blow up and a blow down.


\medskip\noindent
In order to show this result, we apply Corollary \ref{cor: SBLD V} to the morphisms $p_1$ and $p_2$ separately in the following diagram below and analyze how they interact in the effective cone of $\FF_{r,r+1}^{[n]}$. Set,
\[\begin{diagram}
&& \FF_{r,r+1} && \\
& \ldTo^{p_1} & & \rdTo^{p_{2}} &\\
\FF_{r} &&&& \FF_{r+1},
\end{diagram}\]
where $\FF_{r,r+1}$ is the blow up of $\FF_r$ at a point of $E$ (which is isomorphic to the blow up of $\FF_{r+1}$ at a  point not on $E$).

\begin{prop}\label{BLDHirzebruch}
Every linear wall of $\Delta_{\FF_{r+1}^{[n]}}$ either induces a linear wall of $\Delta_{\FF_r^{[n]}}$ for $r>0$ or its induced hyperplane intersects an intersection of two walls.
\end{prop}

\begin{proof}
The strategy is to use prior ideas to lift a curve from $\FF_{r+1}^{[n]}$ to $\FF_{r,r+1}^{[n]}$, and then look at the hyperplane that the image of this curve defines in $N^1(\FF_r^{[n]})$.

Let $\alpha$ be a curve class determining a wall of the SBLD of $\FF_{r+1}^{[n]}$. Then, $\alpha$ can be considered as a curve class $\beta$ in $\FF_{r,r+1}^{[n]}$ since $p_2$ was the blow up at a general point. Let $Z$ be a general scheme (of some component) of those schemes which are in the base loci for all divisors which intersect $\alpha$ negatively and which is not in the base loci for all effective divisors which intersect $\alpha$ non-negatively.
Since the blow up point $p$ was general, then $Z$ is supported away from the strict transform of the fiber containing $p$. 

As such the general point of a general curve $C$ in class $\beta$ corresponds to a scheme whose support misses the exceptional divisor of the map from $\FF_{r,r+1}$ to $\FF_r$ (recall that this is the strict transform of the fiber on $\FF_{r+1}$ through the blow up point).
Thus, an open subset of $C$ is contained in the locus where the map from $\FF_{r,r+1}^{[n]}$ to $\FF_r^{[n]}$ is an isomorphism.
Let $\gamma$ be the class of the closure of this open set in $\FF_r^{[n]}$.

Now, because $Z$ avoids that strict transform of that fiber, we can consider it as a point of $\FF_r^{[n]}$.
Now $Z$ is in the base locus on one side of the induced hyperplane but not in the base locus of some divisor on the other side. This implies the result.
\end{proof}

\medskip\noindent
The converse to the proposition is false, but we do conjecture a partial converse in the case that the number of points is low compared with $r$. 

\begin{conj}\label{CONJ1}
The Hilbert scheme $\FF_{r}^{[r-k]}$ is a Mori dream space, $r > k \geq 0$, and the decomposition of $\EFF(\FF_{r}^{[r-k]})$ is given by the walls defined in Theorem 4.15 of \cite{BC13}. 
\end{conj}

\medskip\noindent
A comment about related work is in order here. In \cite{Hui}, the Brill-Noether divisor induced by a stable vector bundle $E$ over $\PP^2$ is defined.
Divisors of this type fully control the cone of effective divisors, $\EFF(\PP^{2[n]})$.  This conjecture predicts a lower bound for when the Brill-Noether divisors of higher rank vector bundles are needed to fully compute the effective cone.

\medskip\noindent
We also make the following conjecture which would imply that the word linear in the previous proposition is unnecessary. 
\begin{conj}
 $\FF_r$ is a Mori surface.
\end{conj}

\medskip

\medskip
\section{Severi divisors on regular surfaces}\label{SEC4}
\medskip\noindent
The results of the previous sections explore the effective cone and its stable base locus decomposition using birational surfaces. However, these techniques are far from enough to fully describe this decomposition; which in most cases is unknown in general. In this section, we exhibit a set of effective divisors, coming from Severi varieties, in order to know a larger part of the effective cone $\EFF(X^{[n]})$. We explore some properties of such divisors and conclude the section with examples of classes not previously known to be effective in the case of Hirzebruch surfaces. 

\medskip\noindent
Let us consider an effective line bundle $\mathcal{L}$ on a smooth surface $X$, such that the linear system $|\mathcal{L}|$ is not empty. Let $n$ be an integer $0\le n \le p_a(C)$, where $C\in |\mathcal{L}|$ and $p_a$ stands for the arithmetic genus. Let $V_{n}(\mathcal{L})$ be the Severi variety, which generically parametrizes irreducible curves in $|\mathcal{L}|$ with exactly $n$ nodes and no other singularities. This variety $V_{n}(\mathcal{L})$ is a locally closed subscheme of the projective space $|\mathcal{L}|$ and questions about whether it is empty, smooth, or irreducible have been intensely studied and the literature on the subject is vast.

\medskip\noindent
Let us suppose the surface $X$ is regular, $i.e.$ $h^1(\mathcal{O}_X)=0$. This property has strong implications on the Severi variety $V_n(\mathcal{L})$. For example, it implies that the nodes of a curve $C\subset X$ can be smoothed out separately, and this is important information to decide whether $V_n(\mathcal{L})$ is nonempty. In many cases, one can deduce that $V_n(\mathcal{L})$ is not empty by smoothing out nodes of rational curves. For instance, regularity coupled with a result by Chen \cite{CHEN}, implies that the Severi variety $V_n(\mathcal{O}_X(d))\subset |\mathcal{O}_X(d)|$ is not empty and regular (all its components are generically smooth) for $X$ a general quartic in $\PP^3$.

\medskip\noindent
In the case of a rational surface $X$, the Severi variety $V_{n}(\mathcal{L})$ is regular and has the expected dimension if it is nonempty \cite{AC81,T2, SERNESI}. In \cite{T2}, a criteria was shown for deciding whether $V_{n}(\mathcal{L})$ is not empty for rational surfaces. In the case of the plane $\PP^2$ or a Hirzebruch surface $\mathbb{F}_r$, the Severi variety $V_{n}(\mathcal{L})$ is not empty as long as $0\le n \le p_a(C)$, where $C\in |\mathcal{L}|$. Furthermore, it is regular and irreducible \cite{Ha86,Tyom}. We will pay particular attention to these two cases.


\medskip\noindent
Let us exhibit effective divisors on the Hilbert scheme $X^{[n]}$ coming from Severi varieties. Let $V_{n}(\mathcal{L}) \subset |\mathcal{L}|$ be the Severi variety which generically parameterizes irreducible curves in $|\mathcal{L}|$ with exactly $n$ nodes. There is a map $f : V_{n}(\mathcal{L}) \rightarrow X^{[n]}$ sending a curve to the reduced subscheme supported at its nodes. Under some numerical conditions, the image of $f$ will be a divisor in $X^{[n]}$, which we will study when it is not empty. 

\begin{defn}
Let $f$ be the forgetful map $f:V_{n}(\mathcal{L})\rightarrow X^{[n]}.$ We define the \textit{Severi locus}, $\mathrm{Sev}(n,\mathcal{L})$, as the closure of the image of the forgetful map $f$, $$\mathrm{Sev}(n,\mathcal{L})=\overline{\mbox{Im}(f)}.$$ When this locus has codimension one and it is not empty, we will call it the \textit{Severi divisor}.
\end{defn}

\medskip\noindent
We will compute the classes of the Severi divisors, first in the case of the plane $\PP^2$ and then on any regular surface, see Theorem \ref{SEVgen}. We will see that these divisors in many cases realize walls in the stable base locus decomposition of the effective cone  of $\PP^{2[n]}$, see Corollary \ref{IMPOSING}.

\medskip\noindent

\medskip\noindent

\medskip\noindent

\medskip\noindent
\textbf{Severi divisors in $\PP^{2[n]}$. } Let us analyze the case $X = \mathbb{P}^2$ and $|\mathcal{L}|=|\mathcal{O}(d)|$. In this case the Severi variety $V_{n}(\mathcal{L})$ is not empty for $0\le n \le \tfrac{1}{2}(d-1)(d-2)$, it has the expected dimension, and it is irreducible \cite{SERNESI, Ha86}. Observe that, if $\tfrac{1}{6}d(d+3)\le n \le \tfrac{1}{2}(d-1)(d-2)$ (except $(d,n)\neq (6,9)$), then the forgetful map $f$ is birational into its image \cite{TRE1}. Therefore, we have that the nodes of an irreducible curve in $\PP^2$ of degree $d$ form a divisor in $\PP^{2[n]}$, the Severi divisor, when the dimension  $\mbox{dim Im}(f)=2n -1$, which means 
\begin{equation*}
\begin{aligned}
\binom{d+2}{2}&=3n.\\
\end{aligned}
\end{equation*}

\medskip\noindent
In this particular case of $|\mathcal{L}|=|\mathcal{O}(d)|$, we denote $\mathrm{Sev}(n,\mathcal{L})$ simply by $\mathrm{Sev}(n)$.

\medskip
\begin{rmk}
The result of Treger does not hold for reducible nodal plane curves.  For example the family of quintics with seven nodes has a component $V$ consisting of reducible curves with components: an irreducible smooth cubic and a reducible conic. Then, $f(V)$ is not a divisor; even though ${5+2\choose 2}=3\cdot 7$. There is no conflict with our results as in this case the Severi variety $V_{7}(\mathcal{O}(5))$ contains no irreducible curves, hence it is empty. 
\end{rmk}

\begin{prop}\label{SEVCLASS}
Let $\binom{d+2}{2}=3n$, with $d\ge 7$. Then the class of the Severi divisor in $\Pic(\PP^{2[n]})$ is $$\mathrm{Sev}(n)=(3d-3)H-\frac{5}{2}B.$$
\end{prop}

\begin{proof} Observe that $V_n(\mathcal{O}(d))$ is not empty. Let us compute the class of $\mathrm{Sev}(n)$ by intersecting it with two test curves. Let $\gamma_1$ denote the curve induced in $\PP^{2[n]}$ by fixing $p_1,\ldots,p_{n-1}$ points and varying $p_{n}\in l$ on a fixed line $l$. Let $\gamma_2$ be the fiber of the Hilbert-Chow morphism $h:\PP^{2[n]}\rightarrow \PP^{2(n)}$ over a general point of the diagonal.

We need to show  that 
\begin{equation}
\begin{aligned}
\gamma_1\cdot \mathrm{Sev}(n)&=3d-3,\\
\gamma_2\cdot \mathrm{Sev}(n)&=5.
\end{aligned}
\end{equation}

Indeed, the linear system $\Sigma$ of plane curves of degree $d$ nodal at $p_1,\ldots,p_{n-1}$ has dimension 2. After possibly blowing up some extra base points $q_1,q_2 \dots$, it defines a  morphism $f: Y \to \PP^2$ with $Y=\mathrm{Bl}_{p_1,\ldots,p_{n-1}, q_1, \dots}\PP^2$.  The number $\gamma_1\cdot \mathrm{Sev}(n)$ equals $\deg(\pi(R))$ where $\pi: Y \to \PP^2$ is the blow-up morphism and $R$ is the ramification curve of $f$, because $R$ parametrizes the singular points of curves of $\Sigma$ 
(c.f. \cite{EH}). Let $E_i=\pi^{-1}(p_i)$,   $D_j=\pi^{-1}(q_j)$, and $\ell=\pi^* l$. Then $f$ is defined by $L:=d\ell  -2\sum_i E_i- \sum_jk_jD_j$ for some $k_j\ge 1$ and  we have:


\begin{equation*}
\begin{aligned}
c_1(R) &= K_Y-f^*K_{\PP^2}\\ &= -3\ell+\sum_i E_i+\sum_j D_j+3L\\& = (3d-3)\ell- 5\sum_i E_i - \sum_j (3k_j-1)D_j.
\end{aligned}
\end{equation*}

Therefore $\gamma_1\cdot \text{Sev}(n)=\deg(\pi(R)) = 3d-3$. 
The number $\gamma_2\cdot \text{Sev}(n)$ counts the curves $F\in\Sigma$ having $n$ singularities two of which are concentrated  at   $p_{n-1}$.  For this to happen it is necessary and sufficient that the proper transform of  $F$ is  again singular at some point of $E_{n-1}$. Therefore $\gamma_2\cdot \mathrm{Sev}(n)$ 
counts the number of intersections of $R$ with $E_{n-1}$. This number is $5=R\cdot E_{n-1}$.
\end{proof}

\medskip\noindent
{\bf Example:} We set $n=12$, $d=7$. Let us compare the classes, in $\PP^{2[12]}=\langle H, B\rangle$, of the Severi divisor, the extremal divisor of the effective cone $\EFF(\PP^{2[12]})=\langle J,B\rangle$, and the extremal divisor of the movable cone, $\Mov(\PP^{2[12]})=\langle M,H\rangle$. We compute $J$ and $M$ following \cite{Hui,LZ1}. The classes are the following,

\begin{equation*}
\begin{aligned}
J &=7H-B,\\
M &=25H-\tfrac{7}{2}B\sim J +\tfrac{1}{7}H,\\
\text{Sev}(12)&=18H-\tfrac{5}{2}B \sim J  + \tfrac{1}{5}H.
\end{aligned}
\end{equation*}

\medskip\noindent
On the other hand, let us fix a general nonsingular quartic $C$ and consider a general $Z \subset C$, $Z \in \PP^{2[12]}$.  A pencil $P_4 \subset |Z|$ on $C$ defines a curve in $\PP^{2[12]}$ which is moving. It satisfies the following:
$$
P_4\cdot H = 4, \quad P_4\cdot B = 28
$$
and therefore $P_4\cdot J=0$. On the other hand, using the numerical classes, we find that $P_4\cdot \text{Sev}(12) = 2$. 

\medskip\noindent
We can also consider another curve $Q_4\subset \PP^{2[12]}$ by taking a pencil in the linear system $|3K_C|$.  The pencil $Q_4$ is swept on $C$ by a pencil of cubic curves. The curve $Q_4$ is a specialization of $P_4$ and therefore has the same intersection numbers with $H$, $B$, $J$, and $\text{Sev}(12)$ as $P_4$ does. In particular $Q_4\cdot \text{Sev}(12) =2$. There is an apparent contradiction here: no 12-nodal irreducible septic can have its nodes on a cubic, and so one is led to think that $Q_4$ and $\text{Sev}(12)$ have empty intersection; on the other hand $Q_4\cdot \text{Sev}(12) =2$.  This is explained by the fact that $Q_4$ and $\text{Sev}(12)$ can meet (and actually do meet) along the boundary of $\text{Sev}(12)$, i.e. at points of $\text{Sev}(12) \setminus f(V_{12}(7L))$, where $V_{12}(7L)$ stands for the Severi variety of irreducible curves of degree $7$ and precisely $12$ nodes. It would be nice to recognize directly what these two points are.

\medskip\noindent
\begin{cor}\label{IMPOSING}
Let $d\equiv 1\ (\!\!\!\!\mod 5)$ and $d<n\leq \tfrac{1}{3}\binom{d+2}{2}$. Then the class of the Severi divisor  can be written as $\mathrm{Sev}(n)=kH-\tfrac{1}{2}B$, hence it spans a wall of the SBLD.  Furthermore, the augmented and restricted base locus of $\mathrm{Sev}(n)$ contains the schemes which fail to impose independent conditions on curves of degree $k$. 
\end{cor}
\begin{proof}
This is a direct consequence of \cite{ABCH13}.
\end{proof}

\medskip\noindent
{\bf Example:} We set $n=145$, $d=28$.
The class of the Severi divisor is $\mathrm{Sev}(145)=81H-\frac{5}{2}B$, whereas the following class is extremal in the effective cone $J=\frac{576}{37}H-\frac{1}{2}B$. Observe that the class $D_{17}=17H-\frac{1}{2}B$ comes from a well-defined effective divisor and its slope is larger than that of the Severi divisor. This implies that the stable base locus $\mathbf{B}(\mathrm{Sev}(145))$ contains the locus of points that fail to impose independent conditions on curves of degree $17$.

\bigskip\noindent

\medskip\noindent
Let us analyze the case of an incomplete linear system. Let $\Gamma=\{q_1,\ldots,q_r\}$ be a general fixed scheme of degree $r$ and dimension zero. Then consider the Severi variety $V_{n}(\mathcal{L})$, where $|\mathcal{L}|=|\mathcal{O}_{\PP^2}(dH-\Gamma)$.
We have Severi divisors in $\PP^{2[n]}$ as long as 
\begin{equation}\label{SEV}
\begin{aligned}
\binom{d+2}{2}&=3n+r+1,
\end{aligned}
\end{equation}
and the Severi variety is not empty.

\medskip\noindent
For this choice of $|\mathcal{L}|$, the computation of the class of $\mathrm{Sev}(n,\mathcal{L})$ carries over. 

\begin{prop}\label{Prop44}
Let $|\mathcal{L}|\subset|\mathcal{O}_{\mathbb{P}^2}(d)|$ be a subspace of codimension $r$. Assume $\binom{d+2}{2}=3n+r+1$, and that the Severi variety $V_n(\mathcal{L})$ is not empty. Then the class of the Severi divisor in $\Pic(\PP^{2[n]})$ is $$\mathrm{Sev}(n,\mathcal{L})=(3d-3)H-\frac{5}{2}B.$$
\end{prop}
\begin{proof}
Let $\Gamma=\{q_1,\dots ,q_r\}$ be a sufficiently general fixed scheme of degree $r$. We blow up the points $p_1,\dots, p_{n-1},q_1,\ldots q_r\in \PP^2$ and arguing as in Proposition \ref{SEVCLASS} on the linear system $|\mathcal{O}_{\PP^2}(dh-2p_1-\cdots -2p_{n-1}-q_1-\cdots -q_r)|=\PP^2$, the result follows.
\end{proof}

\medskip\noindent
Informally, the previous proposition allows us to interpret the Severi divisor $\mathrm{Sev}(n,\mathcal{L})$ as the locus $Z\in\PP^{2[n]}$ such that $Z$ can be realized as nodes of an irreducible curve of degree $d$ which contains a general fixed subscheme of degree $r$, denoted by $\Gamma$.

\medskip\noindent

\medskip\noindent
{\bf Example:} Set $n=18$, $d=9$ and $r=1$. Let $ |\mathcal{L}|=|\mathcal{O}_{\PP^2}(dH-p)|$, where $p\in \PP^2$ is a fixed point. We can interpret a generic point in $\mathrm{Sev}(18,\mathcal{L})$ as parametrizing configurations of $18$ points that can be realized as nodes of a curve of degree $9$ that contains a fixed point $p\in \PP^2$. In this case, the divisor class $J=23H-\tfrac{5}{2}B$ spans an extremal ray of $\text{Eff}(\PP^{2[18]})$; \cite{Hui}. Moreover, there exist effective divisors $D_k$, for $5\le k\le 17$, whose base loci contain configurations of points that fail to impose independent conditions on curves of degree $k$, \cite{ABCH13}. The class of the Severi divisor is $\mathrm{Sev}(18,\mathcal{L})=24H-\tfrac{5}{2}B$, which implies that there is a containment of base loci $\mathbf{Bs}(D_5)\subset \mathbf{B}(\mathrm{Sev})$. It follows that configurations of $18$ points which fail to impose independent conditions in $|\mathcal{O}_{\PP^2}(5)|$ are in the closure of the Severi divisor $\mathrm{Sev}(18,\mathcal{L})$.

\medskip\noindent


\medskip\noindent
We now consider the case where we only require a subcollection of the points to be the nodes of a curve. We start with an example which later we will generalize.

\medskip\noindent
\textbf{Example:} We want to compute the divisor class $D$ in $\Pic(\PP^{2[13]})$ of the following family. Consider the family of points when we require that $12$ out of a collection of $13$ points are the nodes of an irreducible degree $7$ curve. We will use test curves and the class, in $\PP^{2[12]}$, of the Severi divisor $\mathrm{Sev}(12) = 18H -\frac{5}{2}B$. 

\medskip\noindent
The first test curve $C$ parametrizes twelve fixed general points and a thirteenth point moving on a fixed general line $L$. Let us label the points $p_1,\ldots, p_{12}$ and $p_{13}$, where $p_{13}\in L$ is the moving point. Consider curves which are nodal at $p_1$ through $p_{11}$. Then we know that 
there are $3*7-3=18$ points on $L$ which are the twelfth node of such a curve by Proposition \ref{SEVCLASS}. The same holds true for any collection of eleven points in $p_1$ through $p_{12}$. Consequently $C \cdot D = 12*18 =216$. 

\medskip\noindent
The second test curve $C'$ parametrizes twelve fixed general points and a thirteenth point moving on a fixed line $L$ through one of the previous points. Let us label the points $p_1,\ldots, p_{12}$, and $p_{13}$, where $p_{13}\in L$ is the moving point and  $L$ passes through $p_{12}$. There are now two distinct types of subcollections of the fixed points of size eleven: those with $p_{12}$ and those without $p_{12}$. Consider the only collection without it: $p_1$, $\cdots$, $p_{11}$. Then consider the curves which are nodal at $p_1$ through $p_{11}$. We know that there are $3*7-3=18$ points on $L$ which are the twelfth node of such a curve.
Now, there are eleven collections that include $p_{12}$. For any of these, there are $3*7-3-5=13$ points on the line $L$ which are the twelfth node of such a curve.
Consequently, $C' \cdot D = 18+11*13 =161$. 
Therefore, we get the class $D=216H-\frac{55}{2}B$. An extremal ray of $\EFF(\PP^{2[13]})$ is spanned by $J=\frac{15}{4}H-B/2$.

\medskip\noindent
The previous computation holds in general. The following proposition lists the class for these cases; we omit the details of the proof as no difficulties arise. 

\begin{prop}\label{subCOLLection}
Let $|\mathcal{L}|\subset|\mathcal{O}_{\mathbb{P}^2}(d)|$ be a subspace of codimension $l$ such that $\binom{d+2}{2}=3n+l$, with $d\ge 7$, $n\leq m$. Then the class of the Severi divisor in $\Pic(\PP^{2[m]})$ is $$\mathrm{Sev}(m, \mathcal{L})=\binom{m-1}{n-1}(3d-3)H-\binom{m-2}{n-2}\frac{5}{2}B.$$
\end{prop}
Note, the previous class spans the following ray $$\mathrm{Sev}(m, \mathcal{L})\sim \frac{m-1}{n-1}(3d-3)H-\frac{5}{2}B.$$ 


\subsection{The general computation}

We now prove the main result of this section:

\begin{thm}\label{SEVgen}
Let $\mathcal{L}$ be an effective line bundle over a regular surface $X$. Assume the Severi variety $V_{n}(\mathcal{L})$ has the expected dimension, its generic point parametrizes an irreducible curve with precisely $n$ nodes, and the class $K_X+3C$ is effective. Then the class of the Severi divisor is \[ \mathrm{Sev}(n,\mathcal{L})=(K_X+3C)[n]-\tfrac{5}{2}B[n],\] as long as $\mbox{dim }|C|=3n-1$, or it is empty. 
\end{thm}
\begin{proof} Observe that the expected dimension of the Severi variety is
\begin{equation}\label{DIM}
\mbox{dim}V_{n}(\mathcal{L})=\mbox{dim }|\mathcal{L}|-n.\\
\end{equation}

Then, after blowing up (possibly extra) fixed points, we have a morphism $f:Y\rightarrow \PP^2$, where $Y=Bl_{p_1,\ldots,p_{n-1},q_1,\ldots}X \rightarrow \PP^2$ is given by the linear system $|\mathcal{O}_S(C-p_1-\cdots -p_{n-1})|$. Writing $\pi:Y\rightarrow X$ the blowup map, we can then write the class of the ramification curve $R$ of $f$, as follows $$c_1(R)=\pi^*K_X+3\tilde C- 5\mathbb{E}+ \sum D_j,$$
where $\tilde C$ is the strict transform of the curve $C\subset X$, and  $\mathbb{E}=E_1+\cdots+E_{n-1}$ is the sum of the exceptional divisors over $p_1,\ldots p_{n-1}$ and $D_j$ are the exceptional divisors over $q_j$. Taking $\pi_*c_1(R)$ we get a curve class in $\Pic(X)$ and we can read off the class of the Severi divisor out of it. Observe that since the cohomology class of $\mathrm{Sev}(n,\mathcal{L})\neq 0$, it follows that $\mathrm{dim\ Sev}(n,\mathcal{L}) \ge \mathrm{dim\ }X^{[n]}-1$. However, $\mathrm{dim}V_{n}(\mathcal{L})\le 2n-1$, which implies that $\text{Sev}(n,\mathcal{L})$ is a divisor or empty.
\end{proof}

\medskip\noindent
Note if the Severi variety is reducible, it generically parametrizes irreducible curves means that statement is true for the generic point of each component.

\begin{cor}\label{FINITE}
Under the assumptions of Theorem \ref{SEVgen}, the forgetful map $f:V_n(\mathcal{L})\rightarrow X^{[n]}$
is finite.
\end{cor}


\subsection*{The Severi divisors in $\FF_r^{[n]}$}
Let us now apply these results to Hirzebruch surfaces. Let us denote by $F$ the class of the fiber in $\Pic(\mathbb{F}_r)$ and by $E\subset \mathbb{F}_r$ the unique curve with negative self-intersection.

\medskip\noindent
Let $V_{n}(\mathcal{L}_{a,b})\subset |\mathcal{L}_{a,b}|=|\mathcal{O}_{\FF_r}(aE+bF) \ |$ be the Severi variety parametrizing irreducible curves in $|\mathcal{L}_{a,b}|$ with $n$ nodes and no other singularities. In this case, $V_{n}(\mathcal{L}_{a,b})$ is irreducible if nonempty \cite{Tyom}.
It is nonempty if $\frac{1}{2}(a - 1) (-a r + 2 b - 2)\geq p_a(\mathcal{L}_{a,b})$ and $K_{\mathbb{F}_r}+3C$ is effective.
Then, there is a Severi divisor $\mathrm{Sev}(n)=\overline{f(V_{n}(\mathcal{L}_{a,b}))} \subset \FF_r^{[n]}$ as long as the class $K_{\FF_r}+3C$ is effective and $\mathrm{dim  }V_{n}(\mathcal{L}_{a,b})=2n-1$. This last equation becomes 
\begin{equation}\label{SeveriRelation}
(a+1)(b+1)-\frac{r}{2}(a^2+a)=3n,
\end{equation}
because the Severi variety $V_{n}(\mathcal{L}_{a,b})$ has the expected dimension as long as $b\geq ar$ \cite{SERNESI}. Since $\mathrm{Eff}(\FF_r)=\langle E,F\rangle$, we can assume that $a,b\geq 0$.

\medskip\noindent
Suppose that $(a,b,n,r)$ satisfies (\ref{SeveriRelation}) and $K+3C$ is effective. It follows from Theorem \ref{SEVgen} that the class of the Severi divisor in $\Pic(\mathbb{F}_r^{[n]})=\langle H[n], E[n], B[n] \rangle$ is 
\begin{equation}
\mathrm{Sev}(n,\mathcal{L}_{a,b})=(3a-2)E[n]+(3b-r-2)F[n]-\tfrac{5}{2}B[n].
\end{equation}

\medskip\noindent
\textbf{Example (new effective classes in $\mathbb{F}_1^{[10]}$):} Let us mention the case of $n=10$ points in $\FF_1$. The subcone spanned by $E[10]$, $9F[10]-\frac{1}{2}B[10]$, and $B[10]$ was studied in \cite{BC13}. There is a single line bundle $\mathcal{O}(3H)$ with the correct numbers of sections to give an extremal divisor of the effective cone.

\medskip\noindent
In this case, there are two possible Severi divisors coming from $10$-nodal irreducible curves with classes $3E+8F$ and $4E+7F$, which indeed exist and hence the respective Severi variety is nonempty. These give divisors with the classes $\mathrm{Sev}(10,\mathcal{L}_{3,8})=7E[10]+21F[10]-\frac{5}{2}B[10]$ and $\mathrm{Sev}(10,\mathcal{L}_{4,7})=10E[10]+18F[10]-\frac{5}{2}B[10]$, respectively. Both of these classes are outside the known chambers and outside the span of the four known extremal divisors.

\medskip\noindent
We again can define the Severi divisors for $n$ which do not have the right number of sections. If we append a general collection $n-m$ points to the each point in $\mathrm{Sev}(m)$, we get the divisor class \[\mathrm{Sev}(m)=(3a-2)E[n]+(3b-r-2)F[n]-\tfrac{5}{2}B[n] \in N^1\left(\mathbb{F}_r^{[n]}\right)\] for all $m\leq n$.
Similarly, if we require that a subcollection of $m$ points be in $\mathrm{Sev}(m)$, we get the divisor class \[\mathrm{Sev}(m)=\binom{n-1}{m-1}(3a-2)E[n]+\binom{n-1}{m-1}(3b-r-2)F[n]-\binom{n-2}{m-2}\tfrac{5}{2}B[n] \in N^1\left(\mathbb{F}_r^{[n]}\right)\] for all $m\leq n$. 

\medskip\noindent
\textbf{Example in $\mathbb{F}_1^{[12]}$:} Let us work out the case of $n=12$ points in $\FF_1$. The equation \ref{SeveriRelation} yields $(a+1)(2b-a+2)=72.$
Some solutions to this equation are $(a,b)=\{(7,7),(2,12),(0,35)\}$. Let us consider the first pair $(a,b)=(7,7)$. Theorem \ref{SEVgen} implies that the class of the Severi divisor, which is not empty since Severi variety is nonempty, is
\begin{equation}
\begin{aligned}
\mathrm{Sev}(12,\mathcal{L}_{7,7})=&19E+18F-\tfrac{5}{2}B,\\
=&18H+E-\tfrac{5}{2}B, 
\end{aligned}
\end{equation}
where we write $\Pic(\FF_1^{[12]})=\langle H, E, B\rangle$. Note that $\mathrm{Sev}(12, \mathcal{O}(7))\subset \mathbb{P}^{2[12]}$ has class $18H-\tfrac{5}{2}B$.
This does not conflict with Theorem \ref{MTheo} as imposing nodes on a curve class and pulling back the Severi divisor is distinct from pulling back the curve class and then imposing nodes on that curve class. 

\medskip\noindent
\textbf{Example in $\mathbb{F}_r^{[12]}$:} More generally, the case of $\mathbb{F}_r^{[12]}$ works out similarly. Interestingly, the number and cases of the Severi divisors coming from full linear series depend on the parity of $r$.
If $r=2k+1$ and $k>0$, the Severi varieties come from line bundles with the classes $(7,7k+7)$, $(8,8k+7)$, $(23,23k+12)$, and $(71,71k+35)$. 
If $r=2k$, the Severi varieties come from line bundles with the classes $(5,5k+5)$, $(8,8k+3)$, $(11,11k+2)$, $(17,17k+1)$, and $(35,35k)$. We also have ones coming from the line bundles with class $(3,3k+8)$ if $k>2$.

\medskip\noindent
\medskip\noindent

\medskip\noindent
\subsection*{Severi divisors for some K3 surfaces.}

\medskip\noindent
Let $X$ be one of the following $K3$ surfaces: $S_4\subset \PP^3$, a general quartic surface, or the complete intersection of a general quadric and a cubic, $S_{2,3}\subset \PP^4$ or the complete intersection of 3 general quadrics, $S_{2,2,2}\subset \PP^5$. The Severi variety coming from curves in $|\mathcal{L}_d|=|\mathcal{O}_X(d)|$ has the expected dimension and is regular, if nonempty.

\medskip\noindent
It follows by a result of Chen \cite{CHEN} that in the cases above, the Severi variety $V_n(\mathcal{L}_d)$ is not empty for all $d$, if $0\le n\le p_a(C)$, with $C\in |\mathcal{L}_d|$. Then, by Corollary \ref{FINITE} there exists a Severi divisor in $X^{[n]}$ as long as  dim$V_n(\mathcal{L}_d)=2n-1$. By writing this equation in terms of $d$ and $n$, observe it has integer solutions only in the cases $S_4\subset \PP^3$ and $S_{2,2,2}\subset \PP^5$. Moreover, the case $S_{2,2,2}\subset \PP^5$, admits only one Severi divisor, namely $(d,n)=(1,2)$. That is, two points on $S_{2,2,2}$ form a Severi divisor if they can be realized as nodes of a hyperplane section; a curve of arithmetic genus $5$ and degree 8.

\end{document}